\nonstopmode
\documentclass[11pt]{amsart}
\usepackage[hyperindex=true,bookmarks=true,bookmarksnumbered=true]{hyperref}
\usepackage{amssymb}
\usepackage{graphicx}
\usepackage{verbatim}
\usepackage[all]{xy}
\allowdisplaybreaks

\def\undersetbrace#1\to#2{\underbrace{#2}_{#1}}
\def\oversetbrace#1\to#2{\overbrace{#2}^{#1}}
\def\AMSunderset#1\to#2{\underset{#1}{#2}}
\def\AMSoverset#1\to#2{\overset{#1}{#2}}

\newcommand{\nmb}[2]{\ifx!#1{\ref{nmb:#2}}%
\else\if.#1{\label{nmb:#2}}%
\else\if0#1{\label{nmb:#2}}%
\else{{#2}}%
\fi\fi\fi}
\swapnumbers

\newtheorem*{proposition*}{Proposition}

\newtheorem*{theorem*}{Theorem}
\newtheorem{lemma}[subsection]{Lemma}
\newtheorem*{lemma*}{Lemma}
\newtheorem{corollary}[subsection]{Corollary}
\newtheorem*{corollary*}{Corollary}

\def\idx{}               
\def\ign#1{}             
\def\o{\on{\circ}}

\def\ep{\varepsilon}

\def\ph{\varphi}

\def\ps{\psi}

\def\Ga{\Gamma}

\def\i{^{-1}}
\def\x{\times}
\def\p{\partial}

\let\on=\operatorname
\def\L{\mathcal L}

\def\Fl{\on{Fl}}

\def\Mf{\mathcal{M}f}

\begin{document}
\title[]
{Lie derivatives of sections of natural vector bundles
}
\author{Peter W. Michor}
\address{
Peter W. Michor:
Fakult\"at f\"ur Mathematik, Universit\"at Wien,
Oskar-Morgenstern-Platz 1, A-1090 Wien, Austria.
}
\email{Peter.Michor@univie.ac.at}
\date{\today}
\dedicatory{Dedicated to the memory of Joseph A.\ Wolf}



\begin{abstract} Time derivatives of pullbacks and push forwards along smooth curves of diffeomorphism of sections of natural vector bundles are computed in terms of Lie derivatives along adapted non-autonomous vector fields by extending a key lemma in \cite{MauhartMichor92}. There is also the analogous result about the first non-vanishing derivative of higher order. 
\end{abstract}

\maketitle

\section{Introduction} \label{nmb:1}
The following is an adaptation of the rather well known method of Lie derivation along mapping $N\to M$ as explained for differential forms in \cite[31.11]{Michor08} and more generally for purely covariant tensor fields in \cite[12.2--12.5]{Michor80}. It is used in by now classical proofs of the  Poincar\'e lemma and the theorem of Darboux, see e.g. \cite{Moser65} and \cite{Weinstein71}.
Namely, we prove the the following corollary \ref{nmb:cor1} of   \cite[Lemma 6]{MauhartMichor92}; the need for this result arose during the preparation of \cite{BMV24}.

\begin{corollary*}
Let $\ph_t$ be a smooth curve of local 
diffeomorphisms. Then we get two time dependent vector fields 
$$X_{t_0}=T\ph_{t_0}\i \o \partial_t|_{t_0}\ph_t \text{ and } Y_{t_0} = \p_t|_{t_0}\ph_t \o  \ph_{t_0}\i$$ 
Then for any natural vector bundle 
functor $F$ and for any section $s\in \Ga(F(M))$ we have the 
first non-vanishing derivative
\begin{align*}
\tag{\ref{eq:Lie}}
 \p_t\ph_t^*s &= 
       \ph_t^*\mathcal{L}_{Y_t}s = \mathcal{L}_{X_t}\,\ph_t^*s\,. 
\\ \tag{\ref{eq:Lieinverse}}
\p_t(\ph_t)_*s &= \p_t(\ph_t^{-1})^*s = 
 -(\ph_t)_*\mathcal{L}_{X_t}s = -\mathcal{L}_{Y_t}\,(\ph_t)_*s\,.
\end{align*} 
\end{corollary*}

We also include the analogous result for the first non-vanishing derivative of higher order at fixed time points in Corollary \ref{nmb:cor2}.

\section{Background from \cite{MauhartMichor92}}

\subsection{Curves of local diffeomorphisms}
Let $\ph :\mathbb R\x M\supset U_{\ph }\to M$ be a smooth mapping
where $U_{\ph }$ is an open neighborhood of 
$\{0\}\x M$ in $\mathbb R\x M$, such that
\begin{itemize}
\item $U_\ph\cap (\mathbb R\x \{x\})$ is connected for each $x\in M$,
\item $\ph _t$ is a 
diffeomorphism on its domain $U_{\ph_t}:=U_\ph\cap (\{t\}\x M)$,
\item and $\ph _0=Id_M$. 
\end{itemize}
We say that $\ph_t$ 
is a \idx{\it curve of local diffeomorphisms} though $Id_M$.
 From  Lemma \ref{nmb:lem2} below we see that if 
$\tfrac{\partial^j}{\partial t^j}|_0 \ph_t = 0$ 
for all $1\le j<k$, then 
$X:=\tfrac1{k!} \tfrac{\partial^k}{\partial t^k}|_0\ph_t$ 
is a well defined vector field on $M$. We say that $X$ is the 
\idx{\it first non-vanishing derivative} at 0 of the curve $\ph_t$ of 
local diffeomorphisms. We may paraphrase this as 
$(\partial^k_t|_0\ph_t^*)f=k!\L_Xf$. 

\begin{lemma} \label{nmb:lem2}
Let $c:\mathbb R\to M$ be a smooth curve. If 
$c(0)=x\in M$, $c'(0)=0,\dots,c^{(k-1)}(0)=0$, then $c^{(k)}(0)$ is a 
well defined tangent vector in $T_xM$ which is given by the 
derivation $f\mapsto (f\o c)^{(k)}(0)$ at $x$.

Moreover, if $\ps$ is a local diffeomorphism defined near $x\in M$, then $\ps\o c$ has again first non-vanishing derivative at 0 which is  given by 
$(\ps \o c)^{(k)} = T_x\ps. c^{(k)}(0)$. 
\end{lemma}

\begin{proof} The first claim is  \textrm{\cite[Lemma 2]{MauhartMichor92}}. The second claim follows since tangent vectors at $x$ are mapped by $T_x\ps$.   
\end{proof}

\subsection{Natural vector bundles} \label{nmb:ssec4}
See \cite[6.14]{KolarMichorSlovak93}.
Let $\Mf_m$ denote the category of all
smooth $m$-dimensional manifolds and local diffeomorphisms
between them. A \idx{\it vector bundle functor} or \idx{\it
natural vector bundle} is a functor $F$ which associates a
vector bundle $(F(M),p_M,M)$ to each manifold $M$ and a
vector bundle homomorphism 
$$\xymatrix{
F(M) \ar[r]^{F(f)} \ar[d]^{p_M} & F(N) \ar[d]^{p_N}\\
M  \ar[r]^{f} &    N     
}$$
to each $f:M\to N$ in $\Mf_m$, which covers $f$ and is fiber wise a
linear isomorphism. If $f$ is the embedding of an open subset of $N$ 
then this diagram turns out to be a pullback diagram.
We also point out that $f\mapsto F(f)$ maps 
smoothly parameterized families to smoothly parameterized families, see
\cite[14.8]{KolarMichorSlovak93}. Assuming this property all vector bundle 
functors were classified by \cite{Terng78}: They correspond to linear 
representations of  
higher jet groups, they are associated vector bundles to higher order 
frame bundles, see also \cite[14.8]{KolarMichorSlovak93}. 

Examples of vector bundle functors are
tangent and cotangent bundles, tensor bundles, densities, $M\mapsto L(TM,TM)$, and also the trivial 
bundle $M\x \mathbb R$.

\subsection{Pullback of sections} \label{nmb:ssec5}
Let $F$ be a vector bundle
functor on $\Mf_m$ as described in \ref{nmb:ssec4}. Let $M$ be an 
$m$-manifold  and let $\ph_t$ be a curve of local diffeomorphisms 
through $Id_M$ on $M$. Then
 $\ph_t$, for fixed $t$, is a diffeomorphism defined
on an open subset $U_{\ph_t}$ of $M$.
The mapping
$$\xymatrix{
F(M) \ar[d]^{p_M} &F(U_{\ph_t})  \ar[l]_{\supseteq}  \ar[d]  \ar[r]^{F(\ph_t)\quad}_{\cong\quad} & F(\ph_t(U_{\ph_t})) \ar[r]^{\subseteq} \ar[d] & F(M) \ar[d]^{p_M}
\\
M & U_{\ph_t} \ar[l]_{\supseteq} \ar[r]^{\ph_t\quad}_{\cong\quad} & \ph_t(U_{\ph_t}) \ar[r]^{\subseteq}  &  M    
}$$
is then a local vector bundle isomorphism.

We consider a section $s\in \Ga(F(M))$ of the vector
bundle $(F(M),p_M,M)$ and we define for $t\in \mathbb R$ \emph{pullback} and \emph{push forward} as
$$
\ph_t^*s := F(\ph_{t}\i)\o s\o \ph_t\,,\quad (\ph_t)_* s = (\ph_t\i)^* s = F(\ph_t) \o s \o \ph_t\i\,.
$$
These are local sections of the bundle $F(M)$. 
If $\ph_t$ is smooth curve of diffeomorphisms these are global sections. 
For each $x\in M$
the value $(\ph_t^*s)(x)\in F(M)_x := p_M\i(x)$ is defined, if $t$ is
small enough. So in the vector space $F(M)_x$ the expression 
$\tfrac d{dt}|_0(\ph_t^*s)(x)$ makes sense. These fit together  to a smooth 
section $\tfrac d{dt}|_0(\ph_t)^*s$
which is globally defined and is an element of $\Ga(F(M))$, by the following argument: 

For $x\in M$ there exists $\ep>0$ and an open neighborhoods  $U\subset V$ of $x$ with the closure $\bar U$ compact in $V$ such that $\ph_t^*s$ is a smooth section in $\Ga(F(V))$ for each $t\in (-\ep,\ep)$. Since $(-\ep,\ep)\x V\ni (t,y) \mapsto \ph_t^*s(y)\in F(V)$ is smooth, the curve $t\mapsto\ph_t^*s\in \Ga(F(V))$ is smooth into the Fr\'echet space   $\Ga(F(V))$ by \cite[Lemma 30.8.1]{KM97} with derivative
$\p_t\ph_t^*s \in \Ga(F(V))$.
 
If $\ph_t=\Fl^X_t$ is the flow of a vector field $X$ on $M$  
the section 
$$\L_Xs := \tfrac d{dt}|_0(\Fl^X_t)^*s$$
is called the \idx{\it Lie derivative} of $s$ along $X$. The Lie derivative satisfies  
$\L_X\L_Y-\L_Y\L_X=\L_{[X,Y]}$; see \cite[6.20]{KolarMichorSlovak93}.

\begin{lemma} \label{nmb:lem6} \textrm{\cite[Lemma 6]{MauhartMichor92}} 
Let $\ph_t$ be a smooth curve of local 
diffeomorphisms through $\on{Id}_M$ with first non-vanishing derivative
${k!}X=\partial^k_t|_0\ph_t$. Then for any vector bundle 
functor $F$ and for any section $s\in \Ga(F(M))$ we have the 
first non-vanishing derivative
$${k!}\L_Xs=\partial^k_t|_0\ph_t^*s.$$
\end{lemma}

\section{The results}


\begin{corollary}\label{nmb:cor1}
Let $\ph_t$ be a smooth curve of (local) 
diffeomorphisms. Consider the two time dependent (locally defined) vector fields 
$$X_{t_0}=T\ph_{t_0}\i \o \partial_t|_{t_0}\ph_t \text{ and } Y_{t_0} = \p_t|_{t_0}\ph_t \o  \ph_{t_0}\i$$ 
Then for any vector bundle 
functor $F$ and for any section $s\in \Ga(F(M))$ we have 
\begin{align}
\label{eq:Lie}
 \p_t\ph_t^*s &= 
       \ph_t^*\mathcal{L}_{Y_t}s = \mathcal{L}_{X_t}\,\ph_t^*s\,. 
\\ \label{eq:Lieinverse}
\p_t(\ph_t)_*s &= \p_t(\ph_t^{-1})^*s = 
 -(\ph_t)_*\mathcal{L}_{X_t}s = -\mathcal{L}_{Y_t}\,(\ph_t)_*s\,.
\end{align} 
\end{corollary}

\begin{proof}
Let $\tilde\ph_t = \ph_{t_0}\i \o \ph_{t+t_0}$, a smooth curve of (local) 
diffeomorphisms through $\on{Id}_M$. 
We have 
$$
\p_t|_0 \tilde\ph_t = \p_t|_0 \ph_{t_0}\i \o \ph_{t+t_0} =   T\ph_{t_0}\i \o \p_t|_0\ph_{t+t_0} = T\ph_{t_0}\i \o \p_t|_{t_0}\ph_{t} = X_{t_0}\,.
$$
By Lemma \ref{nmb:lem6} we we get that 
\begin{align*}
\L_{X_{t_0}} s &= \p_t|_{t_0} \tilde\ph_t^*s =  \p_t|_{t_0} (\ph_{t_0}\i \o \ph_t)^*s = \p_t|_{t_0} \ph_t^*(\ph_{t_0}\i)^*s
\\
\implies \quad \L_{X_{t_0}}\ph_{t_0}^*s &= \p_t|_{t_0} \ph_t^*(\ph_{t_0}\i)^*s\quad\text{ which is part of \eqref{eq:Lie}. }
\end{align*}
For the second part of \eqref{eq:Lie} we consider $\bar\ph_t =  \ph_{t+t_0}\o \ph_{t_0}\i$, another smooth curve of local 
diffeomorphisms through $\on{Id}_M$. Here we have, again by Lemma \ref{nmb:lem6}, 
\begin{align*}
\p_t|_0 \bar\ph_t &= \p_t|_0 \ph_{t+t_0} \o \ph_{t_0}\i =    \p_t|_{t_0}\ph_{t} \o \ph_{t_0}\i  = Y_{t_0}\,.
\\
\L_{Y_{t_0}} s &= \p_t|_0 \bar \ph_t^* s = \p_t|_0  (\ph_{t+t_0}\o \ph_{t_0}\i)^* s = \p_t|_0  ( \ph_{t_0}\i)^*\ph_{t+t_0}^* s 
\\&
=  ( \ph_{t_0}\i)^*\p_t|_0 \ph_{t+t_0}^* s \,,
\\&\qquad
\text{ since } ( \ph_{t_0}\i)^*:\Ga(F(M))\to\Ga(F(M))\text{ is bounded linear,}
\\&
= ( \ph_{t_0}\i)^*\p_t|_{t_0} \ph_{t}^* s \text{ which implies the second part of  \eqref{eq:Lie}.}
\end{align*}
To show \eqref{eq:Lieinverse} note first that 
\begin{align*}
  &0 =  \p_t(\on{Id}) = \p_t(\ph_t^{-1}\o \ph_t) = (\p_t\ph_t^{-1})\o \ph_t + T\ph_t^{-1}\o \p_t\ph_t
\\&
\p_t(\ph_t^{-1}) = - T\ph_t^{-1}\o (\p_t\ph_t) \o \ph_t\i
\\&
T\ph_t\o \p_t(\ph_t^{-1}) = -(\p_t\ph_t)\o \ph_t^{-1} = - Y_t
\\&
(\p_t\ph_t^{-1})\o \ph_t = - T\ph_t^{-1}\o (\p_t\ph_t) = -X_t
\end{align*}
Hence, replacing $\ph_t$ by $\ph_t^{-1}$ in \eqref{eq:Lie} replaces $X_t$ by $-Y_t$ and $Y_t$ by $-X_t$ and 
noting that $(\ph_t)_* s = (\ph_t\i)^* s$
transforms \eqref{eq:Lie} into \eqref{eq:Lieinverse}.
\end{proof}

First non-vanishing derivatives of order higher than one makes sense only at discrete time points and not along whole curves of diffeomorphism  which would turn out to be constant. The result is as follows.

\begin{corollary}\label{nmb:cor2}
Let $\ph_t$ be a smooth curve of 
diffeomorphisms which for a fixed time $t_0$ has a first non-vanishing derivative $k!\Xi = \p_t^k|_{t_0} \ph_t$; it is a vector field along the diffeomorphism $\ph_{t_0}$
$$\xymatrix{
& TM \ar[d]^{\pi_M} 
\\
M \ar[ru]^{\Xi} \ar[r]_{\ph_{t_0}} & M  
}.$$
Consider the two (now autonomous) vector fields 
$$X=T\ph_{t_0}\i \o \Xi \text{ and } Y = \Xi \o  \ph_{t_0}\i$$ 
Then for any vector bundle 
functor $F$ and for any section $s\in \Ga(F(M))$ we have the fist non-vanishing derivatives at time $t_0$
\begin{align}
\label{eq:Lie2}
 \p_t^k|_{t_0}\ph_{t}^*s &= 
       k!\ph_{t_0}^*\mathcal{L}_{Y}s = k!\mathcal{L}_{X}\,\ph_{t_0}^*s\,. 
\\ \label{eq:Lie2inverse}
\p_t^k|_{t_0}(\ph_t)_*s &= \p_t^k|_{t_0}(\ph_t^{-1})^*s = 
 -k!(\ph_{t_0})_*\mathcal{L}_{X}s = -k!\mathcal{L}_{Y}\,(\ph_{t_0})_*s\,.
\end{align} 
\end{corollary}


\begin{proof} 
We 
consider  $\tilde\ph_t = \ph_{t_0}\i \o \ph_{t+t_0}$, a smooth curve of local diffeomorphisms through $\on{Id}_M$. By the second part of Lemma \ref{nmb:lem2}  it has the following first non-vanishing derivatives at $t_0$
\begin{align*}
\p_t^k|_0 \tilde \ph_t &= \p_t^k|_0  (\ph_{t_0}\i \o \ph_{t+t_0}) = T\ph_{t_0}\i \o \p_t^k|_{t_0}\ph_t = k!  T\ph_{t_0}\i \o \Xi = k!X\,.
\end{align*}
Then Lemma \ref{nmb:lem6} may be applied as 
follows:
\begin{align*}
k!\L_{X} s &= \p_t^k|_{t_0} \tilde\ph_t^*s =  \p_t^k|_{t_0} (\ph_{t_0}\i \o \ph_t)^*s = \p_t^k|_{t_0} \ph_t^*(\ph_{t_0}\i)^*s
\\
\implies \quad k!\L_{X}\ph_{t_0}^*s &= \p_t|_{t_0} \ph_t^*(\ph_{t_0}\i)^*s\quad\text{ which is part of \eqref{eq:Lie2}. }
\end{align*}
For the second part of \eqref{eq:Lie2} we consider  $\bar\ph_t =  \ph_{t+t_0}\o \ph_{t_0}\i$, another smooth curve through $\on{Id}_M$. As above we have: 
\begin{align*}
\p_t^k|_0 \bar\ph_t &= \p_t^k|_0 \ph_{t+t_0} \o \ph_{t_0}\i =    \p_t^k|_{t_0}\ph_{t} \o \ph_{t_0}\i  = k!\Xi\o \ph_{t_0} = k! Y\,.
\\
k!\L_{Y} s &= \p_t^k|_0 \bar \ph_t^* s = \p_t^k|_0  (\ph_{t+t_0}\o \ph_{t_0}\i)^* s = \p_t^k|_0  ( \ph_{t_0}\i)^*\ph_{t+t_0}^* s 
\\&
=  ( \ph_{t_0}\i)^*\p_t^k|_0 \ph_{t+t_0}^* s \,,
\\&\qquad
\text{ since } ( \ph_{t_0}\i)^*:\Ga(F(M))\to\Ga(F(M))\text{ is bounded linear,}
\\&
= ( \ph_{t_0}\i)^*\p_t^k|_{t_0} \ph_{t}^* s \text{ which implies the second part of  \eqref{eq:Lie2}.}
\end{align*}
To show \eqref{eq:Lie2inverse} note first that for the first non-vanishing derivatives at $t_0$ we have
\begin{align*}
 0 =  \p_t^k|_{t_0}(\on{Id}) &= \p_t^k|_{t_0}(\ph_t^{-1}\o \ph_t) 
  = (\p_t^k|_{t_0}\ph_t^{-1})\o \ph_{t_0} + T\ph_{t_0}^{-1}\o \p_t^k|_{t_0}\ph_t
\\
\p_t^k|_{t_0}(\ph_t^{-1}) &= - T\ph_{t_0}^{-1}\o (\p_t^k|_{t_0}\ph_t) \o \ph_{t_0}\i
\\
T\ph_{t_0}\o \p_t^k|_{t_0}(\ph_t^{-1}) &= -(\p_t^k|_{t_0}\ph_t)\o \ph_{t_0}^{-1} = - Y
\\
(\p_t^k|_{t_0}\ph_t^{-1})\o \ph_{t_0} &= - T\ph_{t_0}^{-1}\o (\p_t^k|_{t_0}\ph_t) = -X
\end{align*}
Hence, replacing $\ph_{t_0}$ by $\ph_{t_0}^{-1}$ in \eqref{eq:Lie2} replaces $X$ by $-Y$ and $Y$ by $-X$ and 
noting that $(\ph_{t_0})_* s = (\ph_{t_0}\i)^* s$
transforms \eqref{eq:Lie2} into \eqref{eq:Lie2inverse}.
\end{proof}

\def\cprime{$'$}

\end{document}